\documentclass[10pt]{amsart}
\usepackage{amsmath}
\usepackage{amsthm}
\usepackage{amsfonts}
\usepackage{amssymb}
\usepackage{hyperref}
\usepackage{blkarray}

\theoremstyle{plain}
\newtheorem{theorem}{Theorem}[section]
\newtheorem{lemma}[theorem]{Lemma}

\newtheorem{remark}{Remark}
\numberwithin{equation}{section}

\newcommand{\ZZ}{\mathbb{Z}}

\newcommand{\wenvert}[1]{\left\lvert\left\lvert#1\right\rvert\right\rvert}
\let\abs=\envert
\let\lm=\lambda
\let\Lm=\Lambda

\DeclareMathOperator{\h}{h}

\begin{document}
\title[On the Ratat-Goormaghtigh equation]
{On the Ratat-Goormaghtigh equation and integer points close to the graph of a smooth function}
\author{Tomohiro Yamada}
\address{Center for Japanese language and culture, Osaka University,
562-8678, 3-5-10, Semba-Higashi, Minoo, Osaka, JAPAN}
\email{tyamada1093@gmail.com}

\date{}

\begin{abstract}
We prove that the sum of reciprocals $1/x$ of integer solutions of $(x^m-1)/(x-1)=N$
with $x, m\geq 2$ for a given integer $N$ except the smallest $x$ is smaller than $5.9037$.
If we limit $x$ to be prime, then the sum is smaller than $0.73194$.
\end{abstract}

\maketitle

\section{Introduction}\label{intro}

The diophantine equation
\begin{equation}\label{eq10}
N=\frac{x^m-1}{x-1}=\frac{y^n-1}{y-1}
\end{equation}
in integers $x, y, m, n, N$ with $x, y\geq 2$ and $m, n\geq 3$
have been studied for more than one hundred years.
Ratat \cite{Rat} noted that $(x, m, y, n, N)=(2, 5, 5, 3, 31)$ satisfies \eqref{eq10}
and asked for further solutions.
Goormaghtigh \cite{Goor} gave another solution $(x, m, y, n, N)=(2, 13, 90, 3, 8191)$
and noted that it is the only solution with $N<10^5$.
Now it is conjectured that these two solutions are the only solutions of \eqref{eq10}
in integers $x, y, m, n, N$ with $x, y\geq 2$ and $m, n\geq 3$.
This conjecture implies that for any given integer $N\geq 2$, the equation
\begin{equation}\label{eq11}
\frac{x^m-1}{x-1}=N
\end{equation}
has at most one solution in integers $(x, m)$ with $x\geq 2$ and $m\geq 3$ except $N=31$ and $N=8191$.

Many results are known for \eqref{eq10}.
However, in this paper, we focus on the distribution of solutions $(x, m)$ of \eqref{eq11}
for a given integer $N>1$.
Loxton \cite{Lox} proved that the number of such pairs is at most $\log^{1/2+o(1)} N$.
Luca \cite{Luc} proved that the number of such pairs $(x, m)$ with $x$ prime
is at most
$$(\log N)^{1/4} \exp \left((c_0+o(1)) \frac{\log\log N}{\log\log\log N}\right)$$
for large $N$.
Moreover, Luca proved that there exist absolute constants $c_1$ and $c_2$ such that
$m<(\log N)^{1/2}\exp(c_1(\log\log N\log\log\log N)^{1/2})$ except at most \\
$c_2(\log\log N/\log\log\log N)^{1/2}$ solutions.

We prove the following explicit result.
\begin{theorem}\label{th1}
Let $(x_i, m_i)$ with $x_i$ and $m_i$ positive integers and $x_1<x_2<\cdots$ be all solutions
$(x, m)$ of \eqref{eq11}.
Then
\begin{equation}
\sum_{i\geq 2}\frac{1}{x_i}<5.9037.
\end{equation}
Moreover, the left hand side of the above inequality tends to zero as $N$ goes to infinity.
\end{theorem}

If we limit $x$ to be prime, then we have the following upper bounds.

\begin{theorem}\label{th2}
Let $(q_i, m_i)$ with $q_i$ prime, $m_i$ positive integers, and $q_1<q_2<\cdots$ be all solutions
$(x, m)$ of \eqref{eq11} with $x$ prime.
Then
\begin{equation}
\sum_{i\geq 2}\frac{1}{q_i}<0.73194
\end{equation}
and
\begin{equation}
\prod_{i\geq 2}\frac{q_i}{q_i-1}<2.07913.
\end{equation}
\end{theorem}

Our basic idea is to count the number of integer solutions $(x, m)$ as integer points of the curve defined by \eqref{eq11}, while Luca's argument is more arithmetic in character.
Clearly \eqref{eq11} is equivalent to
\begin{equation}\label{eq12}
\log (x^m-1)-\log (x-1)=\log N.
\end{equation}
So that, our problem can be restated as the distribution of the number of lattice points $(m, x)$ in the curve defined by \eqref{eq12}.
However, the left hand side function of \eqref{eq12} is rather complicated to analyze geometric properties
of this curve.

Instead, we restate \eqref{eq10} that
\begin{equation}
m\log x-\log (x-1)-\log N=\log\left(\frac{x^m}{x^m-1}\right),
\end{equation}
which implies
\begin{equation}
0<m-\frac{\log (x-1)+\log N}{\log x}<\frac{1}{(x^m-1)\log x}.
\end{equation}
Putting
\begin{equation}\label{eq13}
f_N(x)=\frac{\log N + \log(x-1)}{\log x},
\end{equation}
we have
\begin{equation}\label{eq14}
0<m-f_N(x)<\frac{1}{(x^m-1)\log x}<\frac{1}{N\log x}.
\end{equation}
Now our problem is reduced to study of the distribution of the set
\begin{equation}\label{eq15}
\left\{x\in \ZZ_{\geq 2}: 0<m-f_N(x)<\frac{1}{N\log x}\textrm{ for some integer }m\right\}.
\end{equation}

More generally, for a given real function $f$ and positive real numbers $M$ and $\delta$, we put
$$S(f, M, \delta)=\{x\in\ZZ\cap [M, 2M]: \wenvert{f(x)}<\delta\},$$
where $\wenvert{t}$ denotes the distance of a real number $t$ to the integer $x$ nearest to $t$.
Thus, if $(x, m)$ with $M\leq x\leq 2M$ satisfies \eqref{eq11}, then $x\in S(f_N, M, 1/(N\log M))$.
There is plenty of results on number of the set $S(f, M, \delta)$ such as \cite{Hux}.
A good reference on this topic is \cite[Chapter 5]{Bor}.

\section{Preliminary results}

We use an explicit upper bound given in \cite[Theorem 5.11]{Bor} for the number
$R(f, M, \delta)=\# S(f, M, \delta)$
of integer points close to the graph of a given function $f(n)$.

\begin{lemma}\label{lm21}
Let $k\geq 1$ be an integer and $f(x)$ be a function in $C^k[N, 2N]$.
Assume that there exist constants $\lm>0$ and $c\geq 1$ such that
\begin{equation}\label{eq21}
\lm \leq\abs{f^{(k)}(x)}\leq c\lm
\end{equation}
for $M\leq x\leq 2M$.
Moreover, we put $\alpha=2k(2c)^{2/(k^2+k)}$.

If $(k+1)!\delta<\lambda$, then
\begin{equation}
R(f, M, \delta)\leq \alpha M \lm^{2/(k^2+k)}+4k.
\end{equation}
\end{lemma}

We also use an lower bound for linear forms of logarithms due to Matveev
in order to obtain a lower bound for the size of the second solution $x_2$.
For a rational number $a=\pm n/d$ with $n$ and $d$ positive and mutually prime integers, 
we have the height $\h(a)=\max\{n, d\}$.
We restate \cite[Theorem 2.2]{Mat} in the rational case.

\begin{lemma}\label{lm22}
Let $a_1, a_2, \ldots, a_n$ be rational numbers
and $b_1, b_2, \ldots, b_n$ be arbitrary integers with $b_n\neq 0$.

We put $A_j=\log \h(a_j)$ for $j=1, \ldots, n$,
\begin{equation}
\begin{split}
B= & ~ \max \{1, \abs{b_1}A_1/A_n, \abs{b_2}A_2/A_n, \ldots, \abs{b_n} \},\\
\Omega= & ~ A_1A_2\ldots A_n,\\
C(n)= & ~ \frac{16}{n!}e^n(2n+3)(n+2)(4(n+1))^{n+1} \\
& \times \left(\frac{1}{2}en\right)(4.4n+5.5\log n+7)
\end{split}
\end{equation}
and
\begin{equation}
\Lm=b_1\log a_1+\ldots+b_n\log a_n.
\end{equation}
Then we have $\Lm=0$ or
\begin{equation}
\log\abs{\Lm}>-C(n)\Omega\log (1.5eB).
\end{equation}
\end{lemma}

Among many results on the Ratat-Goormaghtigh equation, we use the following result of \cite{BGM}.

\begin{lemma}\label{lm23}
\eqref{eq10} with $m>n\geq 3$ has no solution other than $(x, y)=(2, 5)$ and $(2, 90)$
in the range $2\leq x<y\leq 10^5$.
\end{lemma}

\section{Properties of $f_N(x)$}

In this section, we shall show an upper bound for integer points close to the graph $y=f_N(x)$.
This requires us to prove some growth properties of $f_N^{(k)}(x)$.

In this section, we assume that $N$ is a real number $\geq 10^{100000}$
(for results in this section, $N$ need not be an integer).

We begin by observing that
\begin{equation}
f_N^{(k)}(x)=\frac{(-1)^k}{\log^{k+1} x}\left(\frac{P_{k, k}(\log x)(\log N+\log(x-1))}{x^k}-\sum_{r=1}^k\frac{P_{k, r}(\log x)}{x^r(x-1)^{k-r}}\right),
\end{equation}
where $P_{k, r}(x)$ with $0\leq r\leq k$ is the polynomial defined by
\begin{equation}
\begin{split}
P_{1, 0}(t)= & ~ P_{1, 1}(t)= 1, \\
P_{k+1, 0}(t)= & ~ ktP_{k, 0}(t), \\
P_{k+1, k+1}(t)= & ~ (kt+k+1)P_{k, k}(t)-tP_{k, k}^\prime(t), \\
P_{k+1, k}(t)= & ~ ((k-1)t+k+1)P_{k, k-1}(t)-tP_{k, k-1}^\prime(t)+tP_{k, k}(t),
\end{split}
\end{equation}
and, for $r=1, \ldots, k-1$,
\begin{equation}
P_{k+1, r}(t)=((r-1)t+k+1)P_{k, r}(t)-tP_{k, r}^\prime(t)+(k-r)tP_{k, r-1}(t).
\end{equation}
For example, we have
\begin{equation}
\begin{split}
P_{k, 0}(t)= & ~ \frac{k!}{\max\{1, k\}} t^k, \\
P_{k, 1}(t)= & ~ \frac{k!}{\max\{1, k-1\}} t^{k-1}, \\
P_{k, 2}(t)= & ~ \frac{k!}{2\max\{1, k-2\}}(t^{k-1}+2t^{k-2}), \\
P_{k, 3}(t)= & ~ \frac{k!}{3!\max\{1, k-3\}} (2t^{k-1}+6t^{k-2}+6t^{k-3}), \\
P_{k, 4}(t)= & ~ \frac{k!}{4!\max\{1, k-4\}} (6t^{k-1}+22t^{k-2}+36t^{k-3}+24t^{k-4}), \\
P_{k, 5}(t)= & ~ \frac{k!}{5!\max\{1, k-5\}}
(24t^{k-1}+100t^{k-2}+210t^{k-3}+240t^{k-4}+120t^{k-5}), \\
P_{k, 6}(t)= & ~ \frac{k!}{6!\max\{1, k-6\}} \\
& \times (120t^{k-1}+548t^{k-2}+1350t^{k-3}+2040t^{k-4}+1800t^{k-5}+720t^{k-6}). \\
\end{split}
\end{equation}

Now we shall show the following estimates of values of $f_N^{(k)}(x)$.

\begin{lemma}\label{lm31}
Let $g_{k, N}(x)=(-1)^k f_N^{(k)}(x)$.
Then, for $1\leq k\leq 6$ and $x\geq 10^5$, we have
\begin{equation}\label{eq31}
\frac{0.999999 P_{k, k}(\log x)\log N}{x^k\log^{k+1} x}<g_{k, N}(x)<\frac{P_{k, k}(\log x)\log N}{x^k\log^{k+1} x}.
\end{equation}
Moreover, we have
\begin{equation}\label{eq32}
\frac{(k-1)!\log N}{x^k\log^2 x}<g_{k, N}(x)<\frac{\tau_k\log N}{x^k\log^2 x},
\end{equation}
for $2\leq k\leq 6$,
where $\tau_k$'s are constants given in Table \ref{tbl1},
and
\begin{equation}\label{eq32b}
\frac{0.999999\log N}{x^k\log^2 x}<g_{1, N}(x)<\frac{\log N}{x^k\log^2 x}.
\end{equation}
\end{lemma}

\begin{table}
\caption{Constants in Lemmas \ref{lm31}-\ref{lm33}}
\begin{center}
\begin{small}
\begin{tabular}{| c | c | c | c |}
\hline
$k$ & $\tau_k$ & $\gamma_k$ & $C_k$ \\
\hline
$1$ & $1$ & $2.24808$ & $0.03022$ \\
$2$ & $1.17372$ & $4.53426$ & $1.04272$ \\
$3$ & $2.56643$ & $9.11515$ & $3.49005$ \\
$4$ & $8.19823$ & $18.2994$ & $6.49141$ \\
$5$ & $34.4344$ & $36.7099$ & $9.57310$ \\
$6$ & $179.227$ & $73.6077$ & $12.5825$ \\
\hline
\end{tabular}
\label{tbl1}
\end{small}
\end{center}
\end{table}

\begin{proof}
We put $g_{k, N}(x)=(-1)^kf_N^{(k)}(x)$.
Thus, we have
\begin{equation}
\begin{split}
& g_{k, N}(x)= \\
& \frac{P_{k, k}(\log x)\log N}{x^k\log^{k+1} x}
-\frac{(k-1)!}{x^k\log^2 x}\left(\left(\frac{x}{x-1}\right)^k \log x-\log(x-1)\right) \\
& +\frac{(Q_k(\log x))\log(x-1)}{x^k\log^{k+1} x}
-\sum_{r=1}^{k-1} \frac{P_{k, r}(\log x)}{x^r (x-1)^{k-r}\log^{k+1} x},
\end{split}
\end{equation}
where $Q_k(t)=P_{k, k}(t)-(k-1)!t^{k-1}$.
Observing that
$$\frac{1}{x}<\log x-\log (x-1)<\frac{1}{x-1}$$
and
$$\frac{k}{x}<\left(\frac{x}{x-1}\right)^k-1<\frac{k+0.5}{x-1}$$
for $k\leq 6$ and $x\geq 10^5$, we have
\begin{equation}
\begin{split}
g_{k, N}(x)<
& ~ \frac{P_{k, k}(\log x)\log N}{x^k \log^{k+1} x}-\frac{(k-1)!(1+k\log x)}{x^{k+1}\log^2 x} \\
& +\frac{Q(\log x)}{x^k\log^k x}-\sum_{r=1}^{k-1} \frac{P_{k, r}(\log x)}{x^k\log^{k+1} x}
\end{split}
\end{equation}
and
\begin{equation}
\begin{split}
g_{k, N}(x)>
& ~ \frac{P_{k, k}(\log x)\log N}{x^k \log^{k+1} x}-\frac{(k-1)!(1+(k+0.5)\log x)}{x^{k+1}\log^2 x} \\
& +\frac{Q_k(\log x)\log(x-1)}{x^k\log^{k+1} x}-\sum_{r=1}^{k-1} \frac{P_{k, r}(\log x)}{(x-1)^k\log^{k+1} x}.
\end{split}
\end{equation}

We shall give a detailed proof for $k=6$.
We observe that
$$\sum_{r=1}^5 P_{6, r}(t)=548t^5+1350t^4+2040t^3+1800t^2+720t=t(P_{6, 6}(t)-120t^5)$$
to obtain, for $N\geq 10^{100000}$ and $x\geq 10^5$,
\begin{equation}\label{eq33}
\begin{split}
g_{6, N}(x)<
& ~ \frac{120\log N}{x^6\log^2 x}-\frac{120(1+6\log x)}{x^7\log^2 x} \\
& +\frac{(P_{6, 6}(\log x)-120\log^5 x)}{x^6 \log^6 x}-\sum_{r=1}^5 \frac{P_{6, r}(\log x)}{x^6 \log^7 x}
\\
= & ~ \frac{P_{6, 6}(\log x)\log N}{x^6\log^7 x}-\frac{120(1+6\log x)}{x^7\log^2 x} \\
< & ~ \frac{P_{6, 6}(\log x)\log N}{x^6\log^7 x}<\frac{180\log N}{x^6\log^2 x}
\end{split}
\end{equation}
and
\begin{equation}\label{eq34}
\begin{split}
g_{6, N}(x)>
& ~ \frac{120\log N}{x^6\log^2 x}-\frac{120(1+6.5\log x)}{x^6(x-1)\log^2 x} \\
& +\frac{(P_{6, 6}(\log x)-120\log^5 x)\log(x-1)}{x^6 \log^7 x}-\left(\frac{x}{x-1}\right)^6\sum_{r=1}^5 \frac{P_{6, r}(\log x)}{x^6 \log^7 x} \\
> & ~ \frac{P_{6, 6}(\log x)\log N}{x^6\log^7 x}-\frac{120(1+6.5\log x)}{x^7\log^2 x}
-\frac{6.5(P_{6, 6}(\log x)-120\log^5 x)}{x^7 \log^6 x}.
\end{split}
\end{equation}
For $x>10^5$, we have
$$120(1+6.5\log x)<\frac{4.56\times 10^{-4} x\log N}{\log x}
<\frac{4\times 10^{-6} P_{6, 6}(\log x)x\log N}{\log^6 x}.$$
Moreover, since $P_{6, 6}(\log x)-120\log^5 x<7851\log^4 x$,
we have
$$6.5(P_{6, 6}(\log x)-120\log^5 x)<2.22\times 10^{-4} x\log^3 x \log N
<\frac{2\times 10^{-6}P_{6, 6}(\log x)x\log N}{\log^2 x}.$$
Now \eqref{eq34} yields that
\begin{equation}\label{eq35}
g_{6, N}(x)>\left(1-\frac{10^{-5}}{\log x}\right)\frac{P_{6, 6}(\log x)\log N}{x^6\log^7 x}
>\frac{0.999999P_{6, 6}(\log x)\log N}{x^6\log^7 x}.
\end{equation}
Now \eqref{eq31} with $k=6$ follows from \eqref{eq33} and \eqref{eq35}
and \eqref{eq32} with $k=6$ also follows observing that $P_{6, 6}(\log x)(1-1/(10^5\log x))>120\log^5 x$.
We can prove \eqref{eq31} and \eqref{eq32} for $k=2, \ldots, 5$ in a similar way.
Finally, \eqref{eq32b} is nothing but the special case $k=1$ of \eqref{eq31}.
\end{proof}

We also use an estimate to the ratio between values of $f_N^{(k)}(x)$.

\begin{lemma}\label{lm32}
Let $M$ be an arbitrary real number $\geq 10^5$.
Then, for any $M\leq x\leq 2M$,
\begin{equation}
0.999999<\frac{f_N^{(k)}(M)}{f_N^{(k)}(x)}<\gamma_k
\end{equation}
for $1\leq k\leq 6$, where $\gamma_k$'s are constants given in Table \ref{tbl1}.
\end{lemma}

\begin{proof}
It follows from \eqref{eq31} that
\begin{equation}
\frac{g_{6, N}(M)}{g_{6, N}(x)}>0.999999
\end{equation}
for $x\geq M$, observing that $P_{6, 6}(\log x)/(x^6\log^7 x)$ is decreasing for $x>10^5$.

We put $R_6(t)=P_{6, 6}(t)/t^5=120+548t^{-1}+\cdots+720t^{-5}$ to have
$$\frac{P_{6, 6}(\log x)}{x^6\log^7 x}=\frac{R_6(\log x)}{x^6\log^2 x}.$$
Now $R_6(t)/R_6(t+\log 2)$ is decreasing for $t>0$ and
\begin{equation}
\frac{P_{6, 6}(\log M)/(M^6\log^7 M)}{P_{6, 6}(\log x)/(x^6 \log^7 x)}
\geq \frac{R_6(\log M)(M^6 \log^2 M)}{R_6(\log (2M))/((2M)^6 \log^2 (2M))}<73.6076
\end{equation}
for $M\leq x\leq 2M$, provided that $M\geq 10^5$.
Combining this and \eqref{eq31}, we have
\begin{equation}
\frac{g_{6, N}(M)}{g_{6, N}(x)}<73.6077
\end{equation}
as desired.
\end{proof}

Now we shall show the following upper bound for integer points close to the graph $y=f_N(x)$.

\begin{lemma}\label{lm33}
Let $M\geq 10^5$ and put $\delta=1/(N\log M)$.
Then, for $k=1, \ldots, 6$, we have $R(f_N, M, \delta)\leq C_k (\log^{2/(k^2+k)} N)M^{1-2/(k+1)}$
with constants $C_k ~ (k=1, \ldots, 6)$ given in Table \ref{tbl1}.
\end{lemma}

\begin{proof}
It immediately follows from Lemma \ref{lm32} that
for $k=1, \ldots, 6$, $f=f_N$ satisfies \eqref{eq21}
with $\lm=f_N^{(k)}(M)/\gamma_k$ and $c=\gamma_k/0.999999$.

Moreover, Lemma \ref{lm31} yields that
\begin{equation}
(k+1)!\delta<\frac{0.999999(k-1)!\log N}{\gamma_k M^k\log^2 M}<\lm<\frac{\tau_k \log N}{\gamma_k M^k \log^2 M}.
\end{equation}

Now, we apply Lemma \ref{lm21} with $\alpha=2k(2\gamma_k/0.999999)^{2/(k^2+k)}$ to obtain
\begin{equation}
R(f_N, M, \delta)\leq
2k \left(\frac{2\tau_k\log N}{0.999999\log^2 M}\right)^{2/(k^2+k)} M^{1-2/(k+1)}+4k,
\end{equation}
which gives the lemma.
\end{proof}

\section{Some diophantine analysis}
As in Theorem \ref{th1},
let $(x_i, m_i) ~ (i=1, 2, \ldots)$ with $2\leq x_1<x_2<\cdots$ and $m_i\geq 2$ be all integer pairs $(x, m)$
satisfying \eqref{eq11}.

In this section, we prove the following lower bound for $x_2$.

\begin{lemma}\label{lm41}
$x_2>\log^{0.33479} N$.
\end{lemma}

\begin{proof}
We begin by observing that
\begin{equation}\label{eq41}
x_2^{m_2}=N(x_2-1)+1\geq N(x_1-1)+1+N=x_1^{m_1}+N.
\end{equation}

We put
$$\Lm=x_1\log m_1-x_2\log m_2-\log\frac{m_1-1}{m_2-1}.$$
Since $(x, m)=(x_1, m_1)$ and $(x_2, m_2)$ both satisfy \eqref{eq11}, we have
\begin{equation}
\Lm=\log\frac{x_1^{m_1}}{x_1^{m_1}-1}-\log\frac{x_2^{m_2}}{x_2^{m_2}-1}.
\end{equation}
From \eqref{eq41}, we can easily see that $x_2^{m_2}>x_1^{m_1}$ and
\begin{equation}\label{eq42}
0<\Lm<\frac{1}{x_1^{m_1}-1}.
\end{equation}

We use Matveev's theorem in the form given in Lemma \ref{lm22} to obtain
\begin{equation}
-\log\abs{\Lm}<C(3)\Omega \log(1.5eB),
\end{equation}
where
$$\Omega=\log x_1 \log (x_2-1) \log x_2, ~ W_0=\log(1.5eB),$$
and
$$B=\max\{m_1 \log x_1/\log x_2, m_2, \log(x_2-1)/\log x_2\}.$$
Since $x_2^{m_2}=N(x_2-1)+1>N(x_1-1)+1=x_1^{m_1}$, we must have $B=m_2$.
Calculation gives $C(3)<1.69019\times 10^{10}$ and therefore
\begin{equation}\label{eq43}
-\log\abs{\Lm}<1.69019\times 10^{10}\log x_1 \log (x_2-1) \log x_2 \log(1.5e m_2).
\end{equation}

Combining \eqref{eq42} and \eqref{eq43}, we have
\begin{equation}
\log N<\log(x_1^{m_1}-1)<1.69019\times 10^{10}\log x_1 \log (x_2-1) \log x_2 \log(1.5e m_2).
\end{equation}
We observe that $m_2<1+\log N/\log 100001$ and
\begin{equation}\label{eq44}
\frac{\log N}{\log 1.5e(1+\log N/\log 100001)}<1.69019\times 10^{10}\log x_1 \log (x_2-1) \log x_2.
\end{equation}

If $N>\exp\exp 34.3882$, then we must have $x_2>\log^{0.33479} N$ since otherwise the right hand side
of \eqref{eq44} is at most
\begin{equation}
1.69019\times 10^{10}\log^2 (x_2-1) \log x_2<\frac{\log N}{\log 1.5e(1+\log N/\log 100001)},
\end{equation}
contrary to \eqref{eq44}.
On the other hand, if $N\leq \exp\exp 34.3882$, then $x_2\geq 100001>\log^{0.33479} N$.
This proves the lemma.
\end{proof}

\section{Proof of Theorem \ref{th1}}

In this section, we prove Theorem \ref{th1}.

If $N<10^{100000}$, then Lemma \ref{lm23} gives $x_2\geq 10^5$ and $m\leq 20000$.
Hence, we obtain
\begin{equation}\label{eq50}
\sum_{i\geq 1}\frac{1}{x_i}\leq \frac{1}{2}+\frac{20000}{x_2}<0.21.
\end{equation}

Thus, we may assume that $N\geq 10^{100000}$.
The quantity $((\log N)/M^k)^{2/(k^2+k)}$ is optimized
when we choose $k=\ell$ such that $\log^{2/\ell} N\leq M\leq \log^{2/(\ell-1)} N$.
We take $M_6=\min\{10^5, \log^{0.33479} N\}$
and put $M_k=\min\{2^n M_6\geq \log^{2/k} N\}$ for $k=0, \ldots, 5$.

Combining Lemmas \ref{lm23} and \ref{lm41} immediately yields that $x_2\geq M_6$.
As noted in the introduction,
it follows from \eqref{eq15} that each $x_i$ in $[M, 2M]$ must belong to \\
$S(f_N, M, 1/(N\log M))$.
Hence, we apply Lemma \ref{lm33} to obtain
\begin{equation}
\sum_{M\leq x_i<2M}\frac{1}{x_i}<\frac{C_k \log^{2/k(k+1)} N}{M^{2/(k+1)}}.
\end{equation}
Summing this over $M=2^n M_k$ with $M_k\leq M<M_{k-1}$, we have
\begin{equation}\label{eq51}
\sum_{M_k\leq x_i<M_{k-1}}\frac{1}{x_i}<\frac{C_k(\log^{2/k(k+1)} N)(M_1^{-2/(k+1)}-M_1^{-2/(k+1)})}{1-2^{-2/(k+1)}}
\end{equation}
for $k=1, \ldots, 6$.

If $N\geq \exp \exp 62.3752$, then, summing \eqref{eq51} over $k=1, \ldots, 6$, we obtain
\begin{equation}\label{eq52}
\begin{split}
\sum_{x_2\leq x_i<\log^3 N}\frac{1}{x_i}
\leq & ~ \frac{70.0333}{\log^{0.048035} N}-\frac{23.6293}{\log^{1/15} N}
-\frac{19.5955}{\log^{1/10} N}-\frac{14.8925}{\log^{1/6} N} \\
& -\frac{9.09791}{\log^{1/3} N}-\frac{2.75742}{\log N}-\frac{0.06044}{\log^2 N} \\
< & ~ 3.0919.
\end{split}
\end{equation}

Now we set $\widetilde M_6=35836.7\log^{1/6} N$ and $\widetilde M_5=3700.84\log^{1/5} N$.
If $\exp\exp 62.3752\geq N\geq \exp \exp 44.9432$, then $M_6\leq \widetilde M_6<M_5$ and, we observe that
\begin{equation}
\sum_{\widetilde M_6\leq x_i<M_5}\frac{1}{x_i}<\frac{C_6(\log^{1/21} N)(\widetilde M_6^{-2/7}-M_5^{-2/7})}{1-2^{-2/7}}
\end{equation}
obtained as \eqref{eq51}.
Summing this with \eqref{eq51} over $k=1, \ldots, 5$, we obtain
\begin{equation}
\begin{split}
\sum_{x_2\leq x_i<\log^3 N}\frac{1}{x_i}
\leq & ~ 
\frac{1}{\log^{0.33479} N}+\log\frac{\widetilde M_6}{\log^{0.33479} N}
+\frac{70.0333N^{1/21}}{\widetilde M_6^{2/7}}-\frac{23.6293}{\log^{1/15} N} \\
& -\frac{19.5955}{\log^{1/10} N}-\frac{14.8925}{\log^{1/6} N}-\frac{9.09791}{\log^{1/3} N}-\frac{2.75742}{\log N}-\frac{0.06044}{\log^2 N} \\
< & ~ 5.0226.
\end{split}
\end{equation}

If $\exp \exp 44.9432\geq N\geq \exp \exp 34.3883$, then, we have $\log^{0.33479} N<M_6\leq \widetilde M_5<M_4$
and as above,
\begin{equation}
\begin{split}
\sum_{x_2\leq x_i<\log^3 N}\frac{1}{x_i}
\leq & ~ 
\frac{1}{\log^{0.33479} N}+\log\frac{\widetilde M_5}{\log^{0.33479} N}+\frac{46.4039N^{1/15}}{\widetilde M_5^{1/3}} \\
& -\frac{19.5955}{\log^{1/10} N}-\frac{14.8925}{\log^{1/6} N}-\frac{9.09791}{\log^{1/3} N}-\frac{2.75742}{\log N}-\frac{0.06044}{\log^2 N} \\
< & ~ 5.90369.
\end{split}
\end{equation}

Now let $\widetilde M=3591050$.
If $\exp\exp 34.3883\geq N\geq \exp(\widetilde M^2)$, then $\widetilde M<M_4$ and 
\begin{equation}
\begin{split}
\sum_{x_2\leq x_i<\log^3 N}\frac{1}{x_i}
\leq & ~ 
\frac{1}{10^5}+\log\frac{\widetilde M}{\log(10^5)}+\frac{46.4039N^{1/15}}{\widetilde M^{1/3}} \\
& -\frac{19.5955}{\log^{1/10} N}-\frac{14.8925}{\log^{1/6} N}-\frac{9.09791}{\log^{1/3} N}-\frac{2.75742}{\log N}-\frac{0.06044}{\log^2 N} \\
< & ~ 5.90359.
\end{split}
\end{equation}
A similar argument allows us to confirm that if $\exp(\widetilde M^2)\geq N\geq 10^{10000}$, then
\\ $\sum_{x_2\leq x_i<\log^3 N}1/x_i<5.90369$.

It is clear that if $x_i\geq \log^3 N$, then $m_i<\log N/(3\log\log N)$.
Hence, we obtain
\begin{equation}\label{eq53}
\sum_{x_i\geq\log^3 N}\frac{1}{x_i}\leq \frac{1}{3\log^2 N\log\log N}<10^{-5}.
\end{equation}
Thus, 
\begin{equation}
\sum_{i\geq 2}\frac{1}{x_i}\leq 5.90369+10^{-5}=5.9037.
\end{equation}
Moreover, from \eqref{eq51} and \eqref{eq53}, we see that the above sum tends to zero as $N$
goes to infinity.
This proves Theorem \ref{th1}.

\begin{remark}
From the proof of Lemma \ref{lm41}, we can easily see that $\log x_2>\log^{1/3-o(1)} N$,
which is enough to show that $\sum_{i\geq 2}1/x_i=o(1)$ as $N$ tends to infinity.
\end{remark}

\section{Proof of Theorem \ref{th2}}

We proceed as in the previous section, although we need more careful argument.
We take a border $W$ for each case.
Then, with the aid of \cite[Theorem 5]{RS}, the sum over prime solutions $q_i$ below $W$ can be bounded as
\begin{equation}\label{eq61}
\sum_{x_2\leq p<W}\frac{1}{p}<\frac{1}{\log^2(\log^{0.33479} N)}
+\left(\log\frac{\log W}{\log(\log^{0.33479} N)}\right),
\end{equation}
where $p$ runs all primes in the range $x_2\leq p<W$.

We begin by putting $W_4=\exp(22.2883)$.
If $\exp\exp 34.3883\leq N\leq \exp\exp 44.5766$, then, 
observing that $M_4\leq W_4<M_3$, we obtain from Lemma \ref{lm33} that
\begin{equation}
\sum_{W_4\leq q_i<M_3}\frac{1}{q_i}<26.8084(\log^{1/10} N)\left(\frac{1}{W_4^{2/5}}-\frac{1}{M_3^{2/5}}\right).
\end{equation}
Combining this with \eqref{eq61} with $W=W_4$, we obtain
\begin{equation}
\begin{split}
\sum_{x_2\leq q_i<\log^3 N}\frac{1}{q_i}
\leq & ~ 
\frac{1}{\log^2 (\log^{0.33479} N)}+\log\frac{W_4}{\log (\log^{0.33479} N)}
+\frac{26.8084 \log^{1/10} N}{W_4^{2/5}} \\
& -\frac{14.8926}{\log^{1/6} N}-\frac{9.09793}{\log^{1/3} N}-\frac{2.75743}{\log N}-\frac{0.06044}{\log^2 N} \\
< & ~ 0.73193.
\end{split}
\end{equation}

If $\exp\exp 44.5766\leq N\leq \exp\exp 67.5537$, then we take $W_5=\exp(27.0215)$,
for which $M_5\leq W_5<M_4$, and proceed as above to obtain
\begin{equation}
\sum_{x_2\leq q_i<\log^3 N}\frac{1}{q_i}<
\sum_{x_2\leq p<W_5}\frac{1}{p}+\sum_{W_5\leq q_i<\log^3 N}\frac{1}{q_i}<0.67057.
\end{equation}
If $\exp\exp 67.5537\leq N\leq \exp\exp 101.848$, then we take $W_6=\exp(34.098)$,
for which $M_6\leq W_6<M_5$, to obtain
\begin{equation}
\sum_{x_2\leq q_i<\log^3 N}\frac{1}{q_i}<
\sum_{x_2\leq p<W_6}\frac{1}{p}+\sum_{W_6\leq q_i<\log^3 N}\frac{1}{q_i}<0.49905.
\end{equation}
If $N\geq \exp\exp 101.848$, then \eqref{eq53} gives
\begin{equation}
\sum_{x_2\leq q_i<\log^3 N}\frac{1}{q_i}<0.49823.
\end{equation}

On the other hand, if $\exp\exp 33.4325\leq N\leq \exp\exp 34.3883$, then
we have $M_4\leq W_4<M_3$ again.
Moreover, $x_2\geq 10^5$ and therefore
\begin{equation}
\begin{split}
\sum_{x_2\leq q_i<\log^3 N}\frac{1}{q_i}
\leq & ~ \frac{1}{\log^2(10^5)}+\log\frac{W_4}{\log(10^5)}
+\frac{26.8084 \log^{1/10} N}{W_4^{2/5}}-\frac{14.8926}{\log^{1/6} N} \\
& -\frac{9.09793}{\log^{1/3} N}-\frac{2.75743}{\log N}-\frac{0.06044}{\log^2 N} \\
< & ~ 0.73193.
\end{split}
\end{equation}
If $\exp\exp 22.2883\leq N\leq \exp\exp 33.4325$, then we put $W_3=\exp(22.2883)$ again to obtain
$M_3\leq W_3<M_2$.
\begin{equation}
\sum_{x_2\leq q_i<\log^3 N}\frac{1}{q_i}<
\sum_{x_2\leq p<W_3}\frac{1}{p}+\sum_{W_3\leq q_i<\log^3 N}\frac{1}{q_i}<0.71332.
\end{equation}
For $10^{100000}\leq N\leq \exp \exp 22.2883$, we obtain
\begin{equation}
\begin{split}
\sum_{x_2\leq q_i<\log^3 N}\frac{1}{q_i}
\leq & ~ \frac{1}{\log^2(10^5)}+\log\frac{\log N}{\log(10^5)}+\frac{2.81787}{\log^{1/3} N}-\frac{2.75743}{\log N}-\frac{0.06044}{\log^2 N} \\
< & ~ 0.66981.
\end{split}
\end{equation}

Hence, if $N\geq 10^{100000}$, then
\begin{equation}
\sum_{x_2\leq q_i<\log^3 N}\frac{1}{q_i}<0.73193.
\end{equation}
Using \eqref{eq53}, we have
\begin{equation}
\sum_{i\geq 2}\frac{1}{q_i}<0.73194,
\end{equation}
which also holds when $N<10^{100000}$ using \eqref{eq50}.
Moreover, we have
\begin{equation}
\prod_{i\geq 2}\frac{q_i}{q_i-1}<\exp\left(0.73194\times \frac{100000}{99999}\right)<2.07913.
\end{equation}
This proves Theorem \ref{th2}.

\section{A related problem}

In Lemma \ref{lm41}, we played with a linear form of three logarithms $\log x_1$, $\log x_2$, and
$\log((x_2-1)/(x_1-1)$.
If these three logarithms are linearly independent over the rationals,
then we would be able to use other bounds such as \cite{Ale} and \cite{MV}.
Now we have a problem when these logarithms are linearly dependent over the rationals.
In other words, which triples $a, b, (b-1)/(a-1)$ with $b>a>1$ integers are multiplicatively dependent?

There is an infinite family $(a, b)=(k^s, k^t(k^s-1)+1)$ for integers $s, t\geq 1$ and $k\geq 2$.
Except them, there are eleven such pairs 
$(a, b)=(3, 4), (4, 9)$, $(6, 16), (6, 81), (15, 36)$, $(16, 25), (16, 81), (40, 625)$,
$(91, 4096), (169, 729)$, and $(280, 15625)$ in the range $a<b\leq 10^5$.

{}
\end{document}